\documentclass[10pt]{article}
\usepackage{amsmath}
\usepackage{amsthm}
\usepackage{amssymb}
\usepackage{verbatim}
\newtheorem{theorem}{Theorem}[section]
\newtheorem{lemma}[theorem]{Lemma}
\newtheorem{proposition}[theorem]{Proposition}
\newtheorem{corollary}[theorem]{Corollary}
\theoremstyle{definition}
\newtheorem{defn}{Definition}
\newtheorem{rmk}{Remark}
\author{Philippe Sosoe \and Percy Wong}

\title{Local semicircle law in the bulk for Gaussian $\beta$-ensemble}
\date{}
\begin{document}
\maketitle
\begin{abstract}
We use the tridiagonal matrix representation to derive a local semicircle law for Gaussian beta ensembles at the optimal level of $n^{-1+\delta}$ for any $\delta > 0$. Using a resolvent expansion, we first derive a semicircle law at the intermediate level of $n^{-1/2+\delta}$; then an induction argument allows us to reach the optimal level.  This result was obtained in a different setting, using different methods, by Bourgade, Erd\"os, and Yau in \cite{BEY} and in Bao and Su \cite{BS}. Our approach is new and could be extended to other tridiagonal models.
\end{abstract}

\section{Introduction}
Given a symmetric $n\times n$ matrix $M$, the Stieltjes transform $s_M$ of its eigenvalue distribution function, $\mu_M$, is defined by:
\begin{align*}
s_s(z)&\equiv \int_{\mathbb{R}}\frac{1}{x-z}\,\mu_M(\mathrm{d}x),\\
\mu_M((-\infty,x]) &=\frac{1}{n}\cdot\sharp\{\lambda \le x: \lambda \text{ is an eigenvalue of } M\}.
\end{align*}
The Stieltjes transform is an important tool in the study of asymptotic properties of random matrices. Several proofs of Wigner's theorem on the eigenvalue distribution of random symmetric matrices in the large $n$ limit are based on identifying a pointwise limit of the Stieltjes transforms with the Stieltjes transform of the semicircle density:
\[\rho_{sc}(\mathrm{d}x) =\frac{2}{\pi}\sqrt{1-x^2}\mathbf{1}_{[-1,1]}\,\mathrm{d}x.\]

Up to a constant, the imaginary part of the Stieltjes transform is the convolution of $\mu_M$ with the Poisson kernel:
\[\Im s_s(z) =\pi P_\eta\ast \mu_s(z)=\pi\int\frac{\eta}{(x-E)^2+\eta^2}\,\mu_M(\mathrm{d}x) ,\quad z=E+i\eta.\]
Hence the imaginary part $\Im s_M$ is a smoothed version of the eigenvalue density, and the imaginary part $\eta=\Im z$ appears as a parameter controlling the resolution of the approximation.
In their work on universality for Wigner matrices (see \cite{ES}, \cite{ESY2}), Erdos et al. developed various versions of the \emph{local} semicircle law. Let $s_n$ be the Stieltjes transform of a normalized Wigner matrix $W$:
\[s_n(z) =\frac{1}{n}\left(n^{-1/2}W-z\right)^{-1}.\]
Then $s_n(z)$ and $s_{sc}(z)$, the Stieltjes transform of the semicircle distribution, are asymptotically arbitrarily close as long as $\Im z \gg n^{-1}$. For precise statements, see \cite{ESY}, \cite{ESY2}. The local semicircle law is a crucial input in proofs of universality for matrices with independent entries. Information on $s$ for $z=E+i\eta$ translates to information on the empirical distribution on scale $\eta$, i.e. control of distribution of the eigenvalue counting functions of intervals of length of order $\eta'>\eta$:
\[\pi\cdot\mu(E-i\eta'/2,E+i\eta'/2)= \int_{E-\eta'/2}^{E+\eta'/2} \Im s_n(E+i\eta)  \,\mathrm{d}E+O(\eta).\]
Away from the edges of the support of the semicircle density, such intervals typically contain $n\eta'$ eigenvalues. See for example Lemma \ref{lem:tv64} below. 

Apart from Wigner matrices, the central objects of interest in random matrix theory are unitarily invariant matrix ensembles, with joint eigenvalue density of the form:
\[\text{const.}\times \prod_{1\leq i<j\leq n} |\lambda_i
- \lambda_j|^\beta
\exp\left(-\sum_1^n V(\lambda_i^2)\right) \prod_i d\lambda_i,\]
with $\beta=1$, $2$, or $4$. For Gaussian potential $V(x)=(\beta/4)\cdot x^2$, this is the eigenvalue density for GOE ($\beta=1$), GUE ($\beta=2$) and GSE (Gaussian Symplectic Ensemble, $\beta=4$) matrices. For general $\beta>0$ and $V$, the formula above defines the $\beta$ \emph{ensemble density}. In \cite{BEY}, the authors compare the local statistics of $n$ ``eigenvalues'' distributed according to the beta ensemble density with a convex, analytic potential $V(x)$ and general $\beta>0$ to the statistics of the Gaussian ensembles. They show that the statistics of $\beta$ ensembles are universal in the sense that they asymptotically match those of the Gaussian ensemble with the same value of $\beta$. Although the proof of universality for $\beta$ ensembles in \cite{BEY} differs from those in the Wigner case in many respects, the initial step is also to derive a fine localization result for ``eigenvalues'' of $\beta$ ensembles with convex analytic potentials. The convexity assumption was subsequently removed in \cite{BEY2}.

The purpose of this paper is to give an alternative derivation of a local semicircle law for Gaussian $\beta$-ensembles. Our approach is based on the tridiagonal matrix model introduced by Dumitriu and Edelman \cite{DE}, and thus essentially different from that in \cite{BEY}. Note that Bao and Su \cite{BS} have derived a local semicircle law, down to scale $1/\sqrt{\log n}$, also based on the tridiagonal representation. Their method does not provide an estimate on the probability of deviation other than that this probability vanishes as $n\rightarrow \infty$.

The main result is the following:
\begin{theorem}
Let $s_{sc}$ be the Stieltjes transform of the semicircle law, and $s$ be the Stieltjes transform of the normalized Gaussian $\beta$-ensemble measure ($V=\frac{\beta}{4} x^2$). Then, for any $c,k, \epsilon >0$ and $0<\delta<1$, there is a constant  $C_{c,k,\epsilon}$ such for $z=E+i\eta$
in the region
\[D_{\epsilon,\delta} := \{z: \Im z > n^{-1+\epsilon}, -1+\delta < \Re z < 1-\delta, |z| \leq 10\},\]
we have:
\[\mathbb{P}(|s_{\text{sc}}(z)-s_{n,\beta}(z)|>c)\le C_{c,k,\epsilon,\delta}n^{-k}.\]
\end{theorem}

The following corollary can be derived from the theorem, see e.g. \cite{BEY}, \cite{ESY}

\begin{corollary}
Let the semiclassical location be defined as the number $\gamma_j$ such that 
\begin{displaymath}
\int_{-\infty}^{\gamma_j}\rho_{sc}(s) ds = \frac{j}{n}
\end{displaymath}
then for any $\delta, \epsilon > 0$, $n > 1$ and any $j \in [\delta n, (1-\delta) n]$,
\begin{displaymath}
|\lambda_j - \gamma_j| < n^{-1+\epsilon}
\end{displaymath}
with probability no smaller than $1-C_{k,\epsilon,\delta}n^{-k}$.
\end{corollary}

\begin{rmk}
The availability of a tridiagonal matrix model is not specific to the Gaussian case and our approach can potentially be extended to give local semicircle laws for other matrix ensembles with tridiagonal representations, such as the Laguerre ensembles considered in \cite{DE}.
\end{rmk}

The entries of the tridiagonal representation of $\beta$ ensemble matrices have very different sizes. The approach based on Schur complementation and predecessor comparison used in \cite{ESY}, \cite{ESY2} as well as further works on the local semicircle law relies on the self-similar structure of Wigner matrices and thus cannot be directly applied in our setting. 

As explained in Section \ref{sec:onehalf}, we proceed by separating the model matrix into its expectation, a deterministic Jacobi matrix whose resolvent can be analyzed using asymptotics for orthogonal polynomials in the complex plane (Section \ref{sec:zero}), and a fluctuation part, whose contribution is shown to be asymptotically negligible. To control the resolvent expansion, we require some estimates for the entries of the resolvent of the expectation matrix; these are established in Section \ref{sec:resolventest}.
The expansion provides us with control of the resolvent down to $\Im z > n^{-1/2+\epsilon}$: an inductive argument in Section \ref{sec:one} establishes the local semicircle law in the region $\Im z> n^{-1+\epsilon}$.

\section{Notation}
\begin{defn}
Let $\beta \geq 1$, the Gaussian $\beta$-ensemble is an ensemble of
eigenvalues in $\mathbb{R}^n$ that have the following probability density:
\begin{equation} \label{eqn:beta}
 df(\lambda_1,\hdots,\lambda_n) = G_{n,\beta} \prod_{1\leq i<j\leq n} |\lambda_i
- \lambda_j|^\beta
\exp\left(-\frac{\beta}{4}\sum_1^n \lambda_i^2\right) \prod_i d\lambda_i, 
\end{equation}
where $G_{n,\beta}$ is a normalization constant.
\end{defn}

\begin{defn}
Given a measure $\mu$ supported on $\mathbb{R}$, the \emph{Stietjes transform} of $\mu$ is defined to be the complex-valued function
\begin{equation}
s_\mu(z) = \int \frac{1}{x-z} d\mu(x).
\end{equation}
The Stietjes transform is analytic on the upper half plane and converges weakly to $\mu(x)$ as $\Im z \rightarrow 0$ .
\end{defn}

\begin{defn}
The \emph{Hermite polynomials} $H_n(x)$ are orthogonal polynomials with respect to the weight $e^{-x^2}$ on the real line (this is the ``physicist's normalization''). The $n$-th Hermite polynomial has leading coefficient $2^n$, and is given by the Rodrigues formula:
\[H_n(x) = (-1)^ne^{x^2}\frac{\mathrm{d}^n}{\mathrm{d}x^n}e^{-x^2}.\]

The $n$-th Hermite function $E_n$ is defined as:
\[E_n(x) = \frac{H_n(\sqrt{2}x)}{2^{(n-1)/2}\sqrt{n}}e^{-x^2/2}.\]
The functions $E_n$, $n\ge 0$ form an orthonormal set in $L^2(\mathbb{R})$.

\end{defn}

Dumitriu and Edelman \cite{DE}, generalizing earlier observations by Trotter \cite{T} in the case $\beta =2 $, introduced the tridiagonal model for the Gaussian $\beta$-ensembles, among other matrix models:

\begin{theorem}
 Consider the matrix given by 
\begin{equation} \label{eqn:mat}
 A_{n,\beta} = \frac{1}{\sqrt{2\beta}}\left( \begin{array}{ccccccc}
               N(0,2) & \chi_{(n-1)\beta} & 0 & \ldots & \ldots & \ldots &  0 \\
	       \chi_{(n-1)\beta} & N(0,2) & \chi_{(n-2)\beta} & 0 & \ldots &
\ldots & 0 \\
		0 & \chi_{(n-2)\beta} & N(0,2) & \chi_{(n-3)\beta}
& 0 & \ldots & 0 \\
		\vdots & \ddots & \ddots & \ddots & \ddots & \ddots & \vdots \\
	      0 & \ldots & 0 & \chi_{3\beta} & N(0,2) & \chi_{2\beta} & 0 \\
	      0 & \ldots & \ldots & 0 & \chi_{2\beta} & N(0,2) & \chi_\beta \\
	      0 & \ldots & \ldots & \ldots & 0 & \chi_\beta & N(0,2)
              \end{array} \right)
\end{equation}
where $N(0,2)$ denotes a random variable whose distribution follows the Gaussian
distribution with mean $0$ and variance $2$; $\chi_k$ is a random variable having a chi distribution with $k$ degrees of freedom.  The upper triangular
part of the matrix consists of independent random variables and the matrix is
symmetric.  The joint density of the eigenvalues of $A_{n,\beta}$ coincides with the
$\beta$-ensemble density.
\end{theorem}

We will refer to $A_{n,\beta}$ as the \emph{symmetric model} for the $\beta$-ensemble. Conjugation by the diagonal matrix
\begin{equation} \label{eqn:diagmat}
D_{n,\beta} \equiv \operatorname{diag}(d_1, \ldots, d_n),
\end{equation}
where
\[d_i(n,\beta) = (\beta n)^{(1-i)/2}\prod_{j=1}^{i-1} \chi_{(n-j)\beta}.\]
(see e.g. \cite{ES}), shows that $df$ is also the joint probability density for eigenvalues of the (non-symmetric) matrix

\begin{equation} \label{eqn:mat2}
 \tilde{A}_{n,\beta} = \frac{1}{\sqrt{2\beta}}\left( \begin{array}{ccccccc}
               N(0,2) & \sqrt{\beta n} & 0 & \ldots & \ldots & \ldots &  0 \\
	       \frac{1}{\sqrt{\beta n}}\chi^2_{(n-1)\beta} & N(0,2) & \sqrt{\beta n} & 0 & \ldots &
\ldots & 0 \\
		0 & \frac{1}{\sqrt{\beta n}}\chi^2_{(n-2)\beta} & N(0,2) & \sqrt{\beta n}
& 0 & \ldots & 0 \\
		\vdots & \ddots & \ddots & \ddots & \ddots & \ddots & \vdots \\
	      0 & \ldots & 0 & \frac{1}{\sqrt{\beta n}}\chi_{3\beta} & N(0,2) & \sqrt{\beta n} & 0 \\
	      0 & \ldots & \ldots & 0 & \frac{1}{\sqrt{\beta n}}\chi_{2\beta} & N(0,2) & \sqrt{\beta n} \\
	      0 & \ldots & \ldots & \ldots & 0 & \frac{1}{\sqrt{\beta n}}\chi_\beta & N(0,2)
              \end{array} \right).
\end{equation}

Throughout the rest of this paper, we will be switching between the symmetric and the asymmetric models and will keep track of the changes involved.

For the rest of the paper, we shall use the following notion of an event depending on some index $n$ having overwhelming probability:
\begin{defn}
We say a sequence of events $E_n$ holds with overwhelming probability if for all $n$, $\mathbb{P}(E_n) \geq 1 - O_C(n^{-C})$ for every constant $C$.
\end{defn}
It should be observed that a union of $n^k$ events of overwhelming probability for some fixed $k>0$ still holds with overwhelming probability.

Lastly, $C>0$ will always denote a constant whose exact value need not concern us.

\section{Local semicircle law for zero temperature}
\label{sec:zero}
Let $n$ be a positive integer, and consider the ``zero-temperature'' $\beta$ ensemble matrix:
\begin{equation} \label{eqn:matzero}
A_{n,\infty}= \frac{1}{\sqrt{2}}\left( 
\begin{array}{ccccc}
0& \sqrt{n-1}& 0 & & \\
\sqrt{n-1}& 0 & \sqrt{n-2} & & \\
& \ddots &\ddots & &\\
& & 0 &  \sqrt{2}& 0\\
& & \sqrt{2} &  0 &1\\
&  & 0 &1& 0
\end{array}
\right).
\end{equation}
This corresponds to the $\beta\rightarrow \infty$ limit of the symmetric model matrices $A_{n,\beta}$.
 
We denote the trace of the resolvent $s_\infty:\mathbb{C}_+ \rightarrow \mathbb{C}_+$ of $A_{n,\infty}$ by
\[s_\infty(z) = \frac{1}{n}\operatorname{tr}\left(\frac{1}{\sqrt{n}}A_{n,\infty}-z\right)^{-1}=\frac{1}{n}\sum_{j=1}^n\frac{1}{\lambda_j -z}.\]
The rescaled resolvent converges uniformly on compact subsets of $z\setminus \mathbb{C}$ (see \cite{deift}, p. 159-167):
\[s_n(z)\rightarrow s_{sc}(z),\]
Here $s_{sc}$ is the Stieltjes transform of the semicircle density:
\[s_{sc}(z) =\frac{2}{\pi}\int_{-1}^{1}\frac{\sqrt{1-x^2}}{x-z}\,\mathrm{d}x.\]
As a first step towards the proof of the local semicircle law for $\beta$ ensembles, we establish the following quantitative version of this result, close to the real axis:
\begin{lemma} \label{thm:zerotemp}
Let $0<\delta<1$. There exists a constant $C_{\delta,\epsilon}$ such that, for any $z$ with $1/10 >\mathfrak{I}z> n^{-1+\epsilon}$, and
$ -1+\delta<\Re z <1-\delta$, we have:
\begin{equation}
\left|s_n(z)-s_{sc}(z)\right|\le C_{\delta,\epsilon}n^{-1}.
\end{equation}
\end{lemma}
The proof makes use of the ``Plancherel-Rotach'' asymptotics for Hermite polynomials in the complex plane obtained by Deift et al. \cite{dkmvz} using the Riemann-Hilbert approach. Their methods yield uniform error bounds which imply our result.

\subsection{Proof of Lemma \ref{thm:zerotemp}}
The eigenvalues $\lambda_j$, $1\le j \le n$ of $A_{n,\infty}$ are given by the $n$ distinct, real zeros of the $n$-th Hermite polynomial $H_n$, rescaled by $\sqrt{2n}$, cf. \cite{D}, \cite{ES}.
For any $z$, the resolvent $s(z)$ can be written as 
\begin{align*}
s_n(z) &= \frac{1}{n}\sum^n_{j=1} \frac{1}{\lambda_j -z} \\
&= -\frac{\left(H_n(\sqrt{2n}z)\right)'}{H_n(\sqrt{2n}z)} = -\sqrt{2n}\cdot \frac{\left(H_n\right)'(\sqrt{2n}z)}{H_n(\sqrt{2n}z)},
\end{align*}
where the prime denotes differentiation. The derivatives of Hermite polynomials satisfy the identity:
\[H_n'(z)= 2nH_{n-1}(z).\]
Hence, the trace of the normalized resolvent of $A_{n,\infty}$ has the expression:
\begin{equation}
s_n(z) = 2^{3/2}n^{1/2}\cdot \frac{H_{n-1}(\sqrt{2n}z)}{H_n(\sqrt{2n}z)}.
\end{equation}
We will check that the final expression is close to the Stieltjes transform of the semicircle density.

In \cite{dkmvz}, the authors derive asymptotics for general orthogonal polynomials with respect to even polynomial weights by means of a rescaled Riemann-Hilbert problem. For $z$ as in the statement of Lemma \ref{thm:zerotemp}, formula (8.32) and Theorem 7.10, and the remarks in Appendix B. in \cite{dkmvz} imply:
\begin{multline*}
\frac{H_n(\sqrt{2n}z)}{2^{3n/2}n^{n/2}}e^{-ng(z)} = \frac{1}{2}\left(\frac{(z-1)^{1/4}}{(z+1)^{1/4}}+\frac{(z+1)^{1/4}}{(z-1)^{1/4}}\right)\\-\frac{e^{-n\varphi(z)}}{2i}\left(\frac{(z-1)^{1/4}}{(z+1)^{1/4}}-\frac{(z+1)^{1/4}}{(z-1)^{1/4}}\right)+O(1/n),
\end{multline*}
with an error term uniform in the region of interest.
The function $g$ is the logarithmic potential of the equilibrium measure associated to the Hermite polynomial:
\[g(z)  = \frac{2}{\pi}\int_{-1}^1 \sqrt{1-x^2}\log(z-x)\,\mathrm{d} x,\]
defined for $z\in \mathbb{C}\setminus (-\infty, 1]$. 
The function $\varphi$ is given by:
\[\varphi(z) = -iz(1-z)^{1/2}(1+z)^{1/2}-2i\arcsin z + i \pi.\]
Here the square roots are principal branches, e.g.:
\[(1-z)^{1/2} =\exp((1/2)\log|1-z|+i(1/2)\arg(1+z)) \]
where $\arg \in (-\pi,\pi)$. The function $\arcsin$ is the inverse of the conformal mapping
\[\sin w : \left\{|\Re w|<\frac{\pi}{2} \right\} \rightarrow \mathbb{C}\setminus (-\infty, -1]\cup[1,\infty). \]
Note that $\sin w$ maps $\{\Im  w > 0\}$ one-to-one and onto the upper half-plane $\{\Im  z >0\}$.

The Riemann-Hilbert problem also provides us with asymptotics for the rescaled $n-1$st polynomial. Indeed, computing the second row of the solution to the Riemann-Hilbert problem and using Stirling's approximation, one finds (cf. (3.9), (5.54) and (8.32) in \cite{dkmvz}):
\begin{multline*}
\frac{H_{n-1}(\sqrt{2n}z)}{2^{(3n-1)/2}n^{(n-1)/2}} e^{-ng(z)}= \frac{1}{2}\left(-\frac{(z-1)^{1/4}}{(z+1)^{1/4}}+\frac{(z+1)^{1/4}}{(z-1)^{1/4}}\right)+\\ \frac{e^{-n\varphi(z)}}{2}\left(\frac{(z-1)^{1/4}}{(z+1)^{1/4}}+\frac{(z+1)^{1/4}}{(z-1)^{1/4}}\right)+O(1/n).
\end{multline*}
The error term is $O(1/n)$ by the uniform boundedness of $(z+1)^{-1/4}$ and $(z-1)^{-1/4}$ for the values of $z$ that concern us.

We now argue that the factor $e^{-n\varphi(z)}$ is rapidly vanishing in $n$.
We have:
\begin{multline}\label{eq: rphi}
\Re  \varphi(z) = \Re z \cdot |1+z|^{1/2}|1-z|^{1/2}\sin\left((1/2)(\arg(1+z)+\arg(1-z))\right)  \\
+ \Im z \cdot |1+z|^{1/2}|1-z|^{1/2}\cos\left((1/2)(\arg(1+z)+\arg(1-z))\right)+ 2\Im \arcsin z.
\end{multline}
The last two terms are always positive. For $\Im  z>0$, $-1+\delta<\Re z <1-\delta$: 
\begin{align*}\arg(1-z) &\in (-\pi/2,0),\\
\arg(1+z) &\in (0,\pi/2).
\end{align*}
Note also that $\Re  \varphi > 0 $ for $\Re z<0$. We will need a lower bound for the real part in the region where $\Re z$ is positive. For such $z$,
\[|z-1|^{1/2}|1+z|^{1/2} \le (1-\Re z)^{1/2}(1+\Re z)^{1/2}\left(1+(\Im  z)^2\right)^{1/2}.\]
On the other hand, since the argument of the sine is negative for $0 < \Re  z <1-\delta$:
\begin{align*}
\sin(\arg(1+z)+\arg(1-z)) &\ge \frac{1}{2}\arctan\left(\frac{\Im  z}{1+\Re  z}\right)-\frac{1}{2}\arctan\left(\frac{\Im  z}{1-\Re  z}\right)\\
&\ge \frac{-\Re  z \cdot \Im  z}{(1+\Re  z)(1-\Re  z)}.
\end{align*}

We compare this to a lower bound for $\Im w=\Im \arcsin z$. By our definition of $\arcsin$:
\begin{align}
\label{eq: iz} \Im  z &=  \frac{e^{\Im  w}-e^{-\Im  w}}{2}\cdot \cos \Re  w,\\
\label{eq: rz} \Re  z &= \frac{e^{\Im  w}+e^{-\Im  w}}{2} \cdot \sin \Re  w.
\end{align}
For $\Re z>0$, this implies $\cos \Re  w, \ \sin \Re  w >0$.
Using (\ref{eq: iz}), (\ref{eq: rz}), we find:
\begin{align*}
\sin \Re  w &\le \Re z\\
\cos \Re w &\ge (1-\Re z)^{1/2}(1+\Re z)^{1/2}
\end{align*}
Using (\ref{eq: iz}) again, this implies
\[\Im w \le (2\delta)^{-1/2} \Im z.\]
Without loss of generality, we may assume $\Im z < \delta/100$.
Now using $e^x=1+e^cx$ for $0<x<c$, we have:
\begin{align*}
2 \le e^{\Im  w}+e^{-\Im  w} &\le 2 + e^{1/10} \Im  w, \\
\Re  z\cdot(1- (e^{1/10}/2)\cdot\Im  w) &\le \sin \Re  w,\\
2\Im w \le e^{\Im  w}-e^{-\Im  w} &\le 2\Im  w +(\Im  w)^2,\\
\cos \Re  w  & \le \left(1-(\Re  z)^2)\right)^{1/2} + 10\Im w\cdot|2\Re  z -e^{2\delta/10}\Im  w|.
\end{align*}
Hence we have:
\[ (1-\Re  z)^{1/2}(1+\Re  z)^{1/2} \Im  w \ge \Im  z +O(\Im w)^{3/2}.\]
Here $O$ stands for some terms of higher order in $\Im  w$ multiplying small constants depending on $\delta$.

Putting all the above together, we find:
\begin{equation*}
\Re  \varphi (z) = \frac{(2-(\Re  z)^2)\cdot \Im  z}{(1-\Re  z)^{1/2}(1+\Re  z)^{1/2}} + O(\Im  z)^{3/2}.
\end{equation*}
Recalling that $\Im  z < \delta/100$, we find
\begin{equation*}
\Re  (n \varphi(z)) > \frac{1}{2\delta^{1/2}} n^\epsilon,
\end{equation*}
uniformly in $\Im  z > n^{-1+\epsilon}$.

The factors 
\[\frac{(z-1)^{1/4}}{(z+1)^{1/4}}\pm\frac{(z+1)^{1/4}}{(z-1)^{1/4}}\]
 are bounded uniformly in the region specified in the theorem, so the above gives  a sub-exponential decay rate for the factors multiplying $e^{-n\varphi(z)}$.
 
Using the approximations above, a calculation shows that, uniformly in $1/10 > \Im  z > n^{-1+\epsilon}$:
 \begin{equation*}s_n(z) = 2 \cdot \left(-z+(z-1)^{1/2}(z+1)^{1/2}\right)+O(1/n).\end{equation*}
We once again take the principal determinations of the square roots. It is readily verified that the boundary values of the imaginary part of the first term on the right for $\mathfrak{I} z \rightarrow 0^+$ are given by $\pi$ times the semicircle density. Since at infinity we have
\[(z-1)^{1/2}(z+1)^{1/2} = z +o(1), \quad z\rightarrow \infty\]
for our choice of the square roots, it follows that the analytic function defined by the expression on the right above is equal to $s_{sc}$.

\section{Semicircle law for Gaussian $\beta$-ensemble at level $n^{1/2+\epsilon}$}
\label{sec:onehalf}
In this section, our goal is to prove the following semicircle law at a suboptimal level of $n^{-1/2+\epsilon}$ for any $\epsilon > 0$:

\begin{proposition} \label{prop:1/2}
Let $s_n(z)$ be the Stieltjes transform of the measure induced by the eigenvalues of the normalized Gaussian $\beta$-ensemble, $\frac{1}{\sqrt{n}}A_{n,\beta}$.  Let $s_{sc}(z)$ be the Stieltjes transform of the semicircle law.  For any $\tau > 0$, there exists a constant $C(\epsilon,\beta)$ independent of $n$ such that, with overwhelming probability, we have:
\begin{equation}
\sup_{z\in D_{\epsilon,\delta}}|s(z)-s_{sc}(z)| \leq C(\epsilon,\beta) n^{-\epsilon/100},
\end{equation}
where for $\epsilon>0$ and $0< \delta <1$ the domain $D$ is defined as
\begin{displaymath}
D_{\epsilon,\delta} := \{z: \delta >\Im z > n^{-1/2+\epsilon}, -1+\delta < \Re z < 1-\delta\}.
\end{displaymath}
\end{proposition}

The idea of the proof is to expand the Green's function around the zero temperature case and estimate the differences between the two.  Using the asymmetric tridiagonal model (\ref{eqn:mat2}), we can write 
\begin{equation}
\frac{1}{\sqrt{n}}\tilde{A}_{n,\beta} = \frac{1}{\sqrt{n}}\tilde{A}_{n,\infty} + \Delta
\end{equation}
where $\Delta= \Delta(n,\beta)$ is a bidiagonal matrix whose entries are independent $N(0,\frac{1}{\sqrt{n\beta}})$ variables on the main diagonal and independent random variables with distributions $\frac{1}{2n\beta}\chi^2_{(n-k)\beta}-\mathbb{E}(\frac{1}{2n\beta}\chi^2_{(n-k)\beta})$ on the main sub-diagonal.

Using the resolvent expansion, we can write
\begin{equation}
\label{eq: resolventexpansion}
\tilde{R}^\beta(z) = \tilde{R}^\infty(z) + \sum_{p=1}^m (\tilde{R}^\infty(z)(-\Delta))^p \tilde{R}^\infty(z) + (\tilde{R}^\infty(z)(-\Delta))^{m+1} \tilde{R}^\beta(z).
\end{equation}
Here, $\tilde{R}^\beta(z) = (\frac{1}{\sqrt{n}}\tilde{A}_{n,\beta} - zI)^{-1}$ is the resolvent matrix for the rescaled asymmetric model (the superscript is not to be confused with taking powers) and we suppress the subscript $n$ when no confusion is likely.  We shall also suppress $z$ in future equations when the dependence on $z$ is understood.  Taking traces on both sides of the previous equation and normalizing by $n^{-1}$, we have
\begin{equation} \label{eq: traceexpansion}
s_\beta(z) = s_\infty(z) + n^{-1}\left(\sum_{p=1}^m\operatorname{tr}((\tilde{R}^\infty(-\Delta))^p\tilde{R}^\infty)\right)+n^{-1}\operatorname{tr}\left((\tilde{R}^\infty(-\Delta))^{m+1} \tilde{R}^\beta\right). 
\end{equation}
The proof of Proposition \ref{prop:1/2} depends on the following estimates on the elements of the resolvent of the symmetric matrix (\ref{eqn:matzero}), $R^\infty$:
\begin{proposition} \label{prop:R}
Let $\epsilon' = \epsilon/8$. For each $1 \le k \le n$ and $\Im z > n^{-1/2+\epsilon}$, we have the estimate
\begin{equation} \label{eqn:Rdiag}
\displaystyle{|R^\infty_{kk}| \leq \min \{ C n^{1/2}(\log n)(1+|\sqrt{k}-\sqrt{n}\Re z|)^{-1}, n^{1/2-\epsilon} \}}.
\end{equation}
For $k$ such that $|\sqrt{k}-E\sqrt{n}|<n^\eta$ with $0<\eta<1/2$, we have
\begin{equation} \label{eqn:Rdiagbigk}
|R_{kk}^\infty|\le C_{E,\eta}n^{1/4-\epsilon'}.
\end{equation}
For $k\neq l$ and $\Im z > n^{-1/2+\epsilon}$
\begin{equation} \label{eqn:Roffdiag}
\displaystyle{|R^\infty_{kl}| \leq n^{1/2-\epsilon'}{k}^{-1/4}{l}^{-1/4} |\sqrt{k}-\sqrt{l}|^{-1}}.
\end{equation}
\end{proposition}

We will first show how to prove Proposition \ref{prop:1/2} from Proposition \ref{prop:R} and then proceed to prove Proposition \ref{prop:R} in the next section.

\begin{proof}[Proof of Proposition \ref{prop:1/2}]
Let us first establish a simple lemma
\begin{lemma}\label{lem:Rdiag}
\begin{equation}
\sum_{k=1}^n |R^\infty_{kk}|^m \leq C n^{\frac{m}{2} - m\epsilon'}
\end{equation}
for $m \geq 3$.
\end{lemma}
\begin{proof}[Proof of lemma]
We split the sum according to the distance $h=h(\Re z) = |\sqrt{k}-\sqrt{n}\Re z|$:
\begin{equation} \label{eqn: hsum}
\sum_k |R^\infty_{kk}|^m = \sum_{k:h\le n^{1/4}}|R_{kk}^\infty|^m +\sum_{k:h> n^{1/4}}|R_{kk}^\infty|^m.
\end{equation}
By \eqref{eqn:Rdiagbigk}, the first term in (\ref{eqn: hsum}) is bounded by:
\[ Cn^{3/4}n^{m/4-m\epsilon'}.\]
To bound the second term, we use (\ref{eqn:Roffdiag}) to find:
\begin{align*}
Cn^{\frac{m}{2}}(\log n)^m \sum_{k:h>n^{1/4}}(1+|\sqrt{k}-\sqrt{n}\Re z|)^{-m} \le Cn^{\frac{m}{2}}(\log n)^m n^{1/2} n^{-m/4}.
\end{align*}
\end{proof}

A term in the resolvent expansion (\ref{eq: resolventexpansion}) is of the form (suppressing the superscript $\infty$) 
\begin{equation} \label{eqn:1termasym}
\frac{1}{n} \sum_{i_1,i_2,\hdots,i_p} \tilde{R}_{i_1 i_2} \Delta_{i_2 i'_2} \tilde{R}_{i'_2 i_3} \Delta_{i_3 i_3'} \hdots \Delta_{i_p i_p'}\tilde{R}_{i'_pi_1}
\end{equation}
where $i'_s$ takes nonzero value only for $i'_s = i_s$ or $i'_s = i_s-1$.  We want to rewrite this sum in terms of the resolvent $R$ of the symmetric model. The asymmetric model matrix is obtained from the symmetric model by the conjugation:
\[\tilde{A}_{n,\beta}=D_{n,\beta}A_{n,\beta}D_{n,\beta}^{-1}.\] 
The transformation rules for the resolvents follow from this; on the diagonal we have
\[R_{kk} = \tilde{R}_{kk}\]
for every $k$, and on the off diagonal:
\begin{equation*}R_{kl} = (n\beta)^{(l-k)/2}\tilde{R}_{kl}\cdot \frac{\prod_{j=1}^{k-1}\chi_{(n-j)\beta}}{\prod_{j=1}^{l-1}\chi_{(n-j)\beta}}.\end{equation*}
As a result, when $i'_k = i_k$ in some summand in (\ref{eqn:1termasym}), the products in the numerator and denominator of the successive factors $\tilde{R}_{i'_{k-1}i_k}$ and $\tilde{R}_{i'_k i_{k+1}}$ cancel each other out. If $i'_k=i_k-1$, after cancellation, there remains a factor of the form:
\[m_{k}=\frac{\chi_{n-k-1}}{\sqrt{n\beta}}.\]
Note that $m_{k}$ is at most of order $O(1)$, with Gaussian tails. Thus we may rewrite (\ref{eqn:1termasym}) as
\begin{equation} \label{eqn:1termsym}
\frac{1}{n} \sum_{i_1,i_2,\hdots,i_p} R_{i_1i_2} m_{i_2}^{i_2-i'_2}\Delta_{i_2i'_2} R_{i'_2i_3} m_{i_3}^{i_3-i'_3}\Delta_{i_3i_3'} \hdots \Delta_{i_pi_p'}R_{i'_pi_1}.
\end{equation}

The offdiagonal entries of $\Delta$ are centered $\chi^2$-square variables with variance
\[\operatorname{Var} \Delta_{k,k-1} = \frac{n-(k-1)}{(2n\beta)^2}=O(1/n).\] 
Since $\Delta_{kl}$ has exponential tails for all $k,l$, with overwhelming probability, we have for any $0<c < \epsilon$, 
\begin{align*}
\max_{k,l} |\Delta_{kl}| &\leq n^{-1/2+c/4}\\
\max_k m_k &\leq n^{c/4}
\end{align*} 
Therefore with overwhelming probability, the sum (\ref{eqn:1termsym}) is bounded by
\begin{displaymath}
n^{-1}n^{-(p-1)/2+(p-1)c/2}\sum_{i_1,\hdots,i_p}|R_{i_1i_2}||R_{i'_2i_3}|\hdots|R_{i'_pi_1}|
\end{displaymath}
We are going to use the estimates from Proposition \ref{prop:R}. The estimates for $R_{kl}$ and $R_{k-1,l}$ only differ by some constant. When $k, k-1\neq l$, this follows immediately from expressions (\ref{eqn:Rdiag}) and (\ref{eqn:Roffdiag}). When $k=l$, notice that the right side of (\ref{eqn:Roffdiag}) is of order $n^{1/2-\epsilon'}$, and we certainly have $|R_{ll}|\le Cn^{1/2-\epsilon}$. Conversely, writing $R_{ll-1}$ as a sum over eigenvectors as in the proof of Proposition \ref{prop:R} in the next section, and using Cauchy-Schwarz and (\ref{eqn:Rdiag}), $|R_{ll-1}|$ can be bounded up to a constant factor by the right side of (\ref{eqn:Rdiag}). Thus, the above sum is bounded with overwhelming probability by 
\begin{equation} \label{eqn:sumR}
C^p n^{-1}n^{-(p-1)/2+(p-1)c/2}\sum_{i_1,\hdots,i_p}|R_{i_1i_2}||R_{i_2i_3}|\hdots|R_{i_pi_1}|,
\end{equation}
for some $C>0$.
To estimate the sum, first notice that we have the following estimate, holding uniformly in $j$:
\begin{equation} \label{eqn:sumoffdiag}
\sum_{k\neq j} \sqrt{k}^{-1}|\sqrt{k}-\sqrt{j}|^{-1} \leq C\log n.
\end{equation}

In the sum, whenever $i_l \neq i_{l+1}$ for some $l$, we gain a power of $n^{-\epsilon'}$ using the estimate (\ref{eqn:Roffdiag}), and from (\ref{eqn:sumoffdiag}), the sum over such pairs introduces a factor bounded by $C\log n$.  Whenever $i_l = i_{l+1}$ and the power $m$ of $|R_{i_li_l}|^m$ is at least $3$, we use lemma (\ref{lem:Rdiag}) to gain a power of $n^{-\epsilon'}$. Due to the repeated index, the power of $i_l$ in the sum is $-1$ and so the sum over $i_l$ will be bounded by $\log n$ as well.  

To make the last paragraph precise, consider the pairs of indices $(i_1,i_2),(i_2,i_3)\hdots,(i_p,i_1)$  as edges on a graph (of $n$ vertices). We call an edge \emph{exploratory (E)} if $i_l \neq i_{l+1}$ and \emph{stationary (S)} otherwise.  Each $p$-tuple of pairs (``path'') belongs to a category $(C_1,C_2,\hdots,C_p)$, where each $C_l \in \{E,S\}$ denotes the type of edge $(i_l,i_{l+1})$.  A category of the form
\begin{displaymath}
(\underbrace{E,E,\hdots,E}_{j_1 \text{times}}, \underbrace{S,S,\hdots,S}_{k_1 \text{ times}},\underbrace{E,\hdots,E}_{j_2 \text{ times}}, \cdots,\underbrace{E,\hdots,E}_{j_K \text{ times}}, \underbrace{S,\hdots,S}_{k_K \text{ times}})
\end{displaymath}
corresponds to a partial sum of (\ref{eqn:sumR}) of the form:
\begin{multline}
\label{eqn:catsum} 
C^p n^{-1}n^{-(p-1)/2+(p-1)c/2} \times
\\ \sum_{i_1,\hdots,i_{j_K+1}}|R_{i_1i_2}|\cdots |R_{i_{j_1} i_{j_1+1} }||R_{i_{j_1+1}i_{j_1+1}}|^{k_1}|R_{i_{j_1+2}, i_{j_1+3}}|\cdots |R_{i_{j_2+j_1+1} i_{j_2+j_1+2}}| \hdots |R_{i_{j_1+\ldots +j_{K}+1} i_{1}}||R_{i_1 i_1}|^{k_K}.
 \end{multline}
The summation is over the set $\{1\le i_1,\ldots,i_{j_K+1}\le n , i_j \neq i_k \text{ for each } j \neq k\}$.
We will bound the contribution to the sum (\ref{eqn:sumR}) from each category using (\ref{prop:R}).
We will sum successively over all indices, beginning with $i_1$. Collecting all factors depending on $i_1$, we find that the sum over this index is:
\[\sum_{i_1\neq i_{j_1+\ldots j_K+1}, i_2} |R_{i_1 i_2}||R_{i_{j_1+\ldots j_K+1},i_1}||R_{i_1i_1}|^{k_K},\]
where $k_K\ge 0$. Inserting the estimates (\ref{eqn:Rdiag}) and (\ref{eqn:Roffdiag}), we that the last sum is bounded by:
\begin{multline}\label{eqn:i1sum} C^3n^{3/2-2\epsilon'} \cdot (\log n) \cdot i_2^{-1/4}\cdot i_{j_1+\ldots j_K+1}^{-1/4} \\ \times \sum_{i_1\neq i_{j_1+\ldots j_K+1}, i_2} i_1^{-1/2}\frac{1}{|\sqrt{i_1}-\sqrt{i_2}}\frac{1}{|\sqrt{i_1}-\sqrt{i_{j_1+\ldots j_K+1}}|}\frac{1}{(1+|E\sqrt{n}-\sqrt{i_1}|)^{k_K}}.\end{multline}
Note that $|\sqrt{i_1}-\sqrt{i_{j_1+\ldots j_K+1}}|\ge C i_{j_1+\ldots j_K+1}^{-1/4}$. We use this bound, and perform the sum over $i_1$ using (\ref{eqn:sumoffdiag}); when $k_K\le 2$, we obtain:
\[ C^4 n^{3/2-2\epsilon'} \cdot (\log n)^2 \cdot i_2^{-1/4}\cdot i_{j_1+\ldots j_K+1}^{1/4}.\]
When $k_K\ge 3$, performing a dyadic decomposition around $E\sqrt{n}$ as in the proof of Lemma $\ref{lem:Rdiag}$, and using (\ref{eqn:sumoffdiag}) on each dyadic piece, we obtain the bound
\[ C^4 n^{3/2-(2+k_K)\epsilon'} \cdot (\log n)^2 \cdot i_2^{-1/4}\cdot i_{j_1+\ldots + j_K+1}^{1/4}. \]
We then sum successively over $i_2, i_3, \ldots, i_{j_1+\ldots + j_K}$. For each index, we encounter a sum of the form:
\[ \sum_{i_l \neq i_{l+1}} i_l^{-1/4}|R_{i_l i_l}|^k|R_{i_li_{l+1}}|,\]
where $k \ge 0$. Using (\ref{eqn:Rdiag}), (\ref{eqn:Roffdiag}), and (\ref{eqn:sumoffdiag}) as above, this is bounded by:
\[C^{k+1} i_{l+1}^{-1/4} n^{(k+1)/2}n^{-\epsilon'}\log n,\]
if $0\le k \le 2$, and
\[C^{k+1}  i_{l+1}^{-1/4} n^{(k+1)/2}n^{-(k+1)\epsilon'}\log n,\]
whenever $k\ge 3$.
When we reach the final index $i_{j_1+\ldots + j_K+1}$, after cancelling the factor $i_{j_1+\ldots j_K+1}^{1/4}$ carried over from the summation over $i_1$, we have to sum
\[\sum_{i_{j_1+\ldots + j_K+1}} |R_{ i_{j_1+\ldots + j_K+1} i_{j_1+\ldots + j_K+1} }|^{k_{K-1}}.\]
When $k_{K-1}= 0$, this is of size $n$. When $k_{K-1}\ge 3$, it is bounded by $n^{1/2-k_{k_K-1}\epsilon'}$. 

In summary, for any block of three consecutive edges in any category, we gain a power of $-\epsilon'$. Multiplying the contributions from the summations over the $p$ indices, we find that the expression in (\ref{eqn:catsum}) is bounded by
\[C^{2p} n^{-(p+1)/2+(p-1)c/2} \cdot n^{(p+1)/2}n^{-p\epsilon'/3}.\]
The last expression is $O(n^{-p\epsilon/48})$ provided $c$ is chosen sufficiently small. Since there are at most $2^p$ categories contributing to (\ref{eqn:sumR}), all terms in the first sum in the expansion (\ref{eq: traceexpansion}) vanish as $n\rightarrow \infty$. Taking $m$ in (\ref{eq: resolventexpansion}) large enough, depending on $\epsilon$, to make $p\epsilon/48) < -1/2$ when $p>m$, we can bound the remainder term of the resolvent expansion using the trivial bound for $R^\beta_{ij}$: $|R^\beta_{ij}| \le n^{1/2-\epsilon}$.

Returning to the sum in the expansion (\ref{eq: traceexpansion}), the above implies that 
\begin{equation*}
s_\beta(z) = s_\infty(z) + O(n^{-\epsilon/100}).
\end{equation*}
The proposition then follows from the above estimate and Lemma \ref{thm:zerotemp}.
\end{proof}

\section{Proof of estimates for $R_{kl}$}
\label{sec:resolventest}
In this section, we prove the estimates for $R^\infty_{kl}$ in Proposition \ref{prop:R}, reproduced here for the reader's convenience:

For $\Im z > n^{-1/2+\epsilon}$, we have
\begin{equation*}
\displaystyle{|R^\infty_{kk}| \leq \min \{ C n^{1/2}(\log n)(1+|\sqrt{k}-\sqrt{n}\Re z|)^{-1}, n^{1/2-\epsilon} \}};
\end{equation*}

\begin{equation*}
|R^\infty_{kk}|\le Cn^{1/4-\epsilon'},
\end{equation*}
provided $|\sqrt{k}-E\sqrt{n}|\le n^\eta$ for some $0<\eta<1/2$, and
\begin{equation*}
\displaystyle{|R^\infty_{kl}| \leq n^{1/2-\epsilon/8}{k}^{-1/4}{l}^{-1/4} |\sqrt{k}-\sqrt{l}|^{-1}}
\end{equation*}
for $k\neq l$ and $\Im z > n^{-1/2+\epsilon}$.

The proof relies on the well-known Plancherel-Rotach asymptotics for Hermite polynomials. These comprise three asymptotic expressions, corresponding to the the behavior of $E_k$ in three regions defined relative to $\pm\sqrt{2k}$, the order of magnitude of the largest zeros of $H_k$. In the case of Hermite polynomials, a classical reference is \cite{S}, Chapter 8. In \cite{dkmvz}, analogous formulas are derived for a general class of orthogonal polynomials.

\begin{theorem} \label{thm:PR}
Let $E_k$ be the $k$-th Hermite function, then for any $0<\mu<1$ and $x$ in the ``oscillatory region'':
\[\{x:\sqrt{k}(-1+\mu) < x < \sqrt{k}(1-\mu)\},\] we have the following asymptotic formula ((2.20) in \cite{dkmvz}; see also (8.22.12) in \cite{S}):
\begin{equation} \label{eqn:oss}
\begin{split}
E_k(x) &= \sqrt{\frac{2}{\pi}} k^{-1/4}\left(1-\frac{x^2}{k}\right)^{-1/4} \cos\left(\sqrt{k}\pi \int_{\sqrt{k}}^x \rho_{sc}(y/\sqrt{k}) \,\mathrm{d}y + \frac{1}{2}\arcsin(x/\sqrt{k})\right) \left(1+O\left(\frac{1}{k}\right)\right)\\
& + \sin\left(\sqrt{k}\pi \int_{\sqrt{k}}^x \rho_{sc}(y/\sqrt{k})\, \mathrm{d}y - \frac{1}{2}\arcsin(x/\sqrt{k})\right)O\left(\frac{1}{k}\right)
\end{split}
\end{equation}
where $\rho_{sc}$ denotes the semicircle density on the interval $(-1,1)$.

For $x$ in the ``transition region'': 
\[\{x: \sqrt{k}(1-\mu) < x < \sqrt{k}(1+\mu)\},\] 
we have the uniform asymptotics (\cite{dkmvz}, (2.21) and (2.23); see also (8.22.14) in \cite{S}):
\begin{equation} \label{eqn:trans}
\begin{split} 
E_k(x) &= k^{-1/4}\Big(-\left(1+\frac{x}{\sqrt{k}}\right)^{1/4}\left|1-\frac{x}{\sqrt{k}}\right|^{-1/4}\left(f_k(x/\sqrt{k})\right)^{1/4}Ai\left(f_k(x/\sqrt{k})\right)\\
&+\left|1-\frac{x}{\sqrt{k}}\right|^{1/4}\left(1+\frac{x}{\sqrt{k}}\right)^{-1/4}\left(f_k(x/\sqrt{k})\right)^{-1/4}Ai'\left(f_k(x/\sqrt{k})\right)\Big)\left(1+O\left(\frac{1}{k}\right)\right),
\end{split}
\end{equation}
when $\sqrt{k}(1-\mu) <  x\le 1$. If $1 < x < \sqrt{k}(1-\mu)$, the same asymptotic formula holds with the overall sign reversed. Here $Ai$ is the Airy function and $Ai'$ its derivative and 
\begin{displaymath}
(-f_k(x))^{3/2} = \begin{cases} 
 -k\frac{3\pi}{2}\int_1^x \rho_{sc}(y)\,\mathrm{d}y& 1-\mu <x\le1.\\
 k\frac{3\pi}{2}\int_1^x \sqrt{x-1}\sqrt{1+x}\,\mathrm{d}y& 1<x<1+\mu.
 \end{cases}
\end{displaymath}
Similar asymptotics hold in the transition region around the left edge:
\[ \{x: \sqrt{k}(-1-\mu) < x < \sqrt{k}(-1+\mu)\}.\]
Finally outside the transition region, for $|x|>\sqrt{k}$, the Hermite functions have Gaussian decay.
\end{theorem}
From the alternate form of the asymptotics in the transition region found in \cite{S}, (see 8.22.14 there) combined with (\ref{eqn:oss}) it follows that:
\[|E_k(x)|\le Ck^{-1/12},\]
for some $C$ independent of $k$.

We shall also use the following asymptotics for the location of the eigenvalues of the zero temperature case, see \cite{dkmvz}, Theorem 2.29 :
\begin{theorem}
Let $\bar{\lambda}_{k,n}$ be the $k$-th zero of the $n$-th Hermite polynomial, where $k_0 \leq k \leq n-k_0$ for some $k_0$, then we have:
\begin{equation} \label{eqn:zeroloc}
\left|\bar{\lambda}_{k,n}-\zeta\left(\frac{6k-3}{6n}+\frac{1}{2\pi n}\arcsin(\zeta(k/n))\right)\right| \leq \frac{C}{n^2(\frac{k}{n}(1-\frac{k}{n}))^{4/3}}
\end{equation}
where $\zeta$ is the inverse function to $x\mapsto \int_x^1 \rho_{sc}(y)\,\mathrm{d}y$.
\end{theorem}

In this section and the next, we shall use heavily the fact that the eigenvector matrix for $A_{n,\infty}$ is given by columns of the form 
\begin{displaymath}
\displaystyle{\frac{1}{\sqrt{n}E_{n-1}(\lambda_m)}(E_{n-1}(\lambda_m),E_{n-2}(\lambda_m),\hdots,E_{1}(\lambda_m),E_0(\lambda_m))^t}
\end{displaymath}
where $\lambda_m$ is the corresponding eigenvalue and $E_k$ is the $k$-th Hermite function.  The proof of this fact uses the three-term recurrence relation for Hermite polynomials. It can be found in \cite{D} or \cite{ES}.  We shall also adopt the convention that the eigenvalues $\lambda_m$ are scaled to be of order $\sqrt{n}$ and $\bar{\lambda}_m$ to be the scaled version that are of order $1$.

\begin{proof}[Proof of Proposition \ref{prop:R}]
We will first establish the bounds for the diagonal entries $R_{kk}$:
\begin{equation} \label{eqn:Rdiag1}
|R_{kk}| \leq C n^{1/2} (\log n) \cdot (1+|\sqrt{k}-\sqrt{n}\Re z|)^{-1}.
\end{equation}
First we write
\[R_{kk}=\sum_m \frac{u_m(k)^2}{\bar{\lambda}_m-z}\]
where $u_m$ is the eigenvector corresponding to the eigenvalue $\lambda_m$, since $u_m(k) = \frac{E_k(\lambda_m)}{\sqrt{n}E_{n-1}(\lambda_m)}$ the above expression is
\[R_{kk} = \sum_m \frac{E_k(\lambda_m)^2}{nE_{n-1}(\lambda_m)^2}\frac{1}{\bar{\lambda}_m-z}\]
For the denominator, we have the following asymptotics \cite{tao}, section 2.6.4:
\begin{equation}
\label{eqn:En-1}
nE_{n-1}(\lambda_m)^2 = n^{1/2}(\rho_{sc}(\bar{\lambda}_m) + O(n^{-1})),
\end{equation}
provided $|\bar{\lambda}_m|< 1-9\delta/10$.
Since $-1+\delta<\Re z<1-\delta$, we have:
\begin{align} 
|R_{kk}| \le&  C_\delta n^{-1/2} \sum_{|\bar{\lambda}_m|< 1-9\delta/10} \frac{E_k(\lambda_m)^2}{|\bar{\lambda}_m-z|}+ C_\delta \sum_{|\bar{\lambda}_m|\ge 1-9\delta/10}u_m(k)^2 \nonumber\\
\leq& C_\delta n^{-1/2} \sum_{|\bar{\lambda}_m|< 1-9\delta/10} \frac{E_k(\lambda_m)^2}{|\bar{\lambda}_m-z|} + C_\delta. \label{eqn:Rdiagtemp}
\end{align}
By the Plancherel-Rotach asymptotics (Theorem \ref{thm:PR}), we have
\begin{equation} \label{eqn:hermdecay}
E_k(x) \leq C\min\left\{k^{-1/4}\left(1-\frac{x^2}{k}\right)^{-1/2}, k^{-1/12}\right\}.
\end{equation}
Let $E = \Re z$ , without loss of generality, we shall assume that $E \geq 0$, and that $|E - \sqrt{k/n}| = h/\sqrt{n}$, i.e. 
\[|\sqrt{k}-\sqrt{n}E| = h,\]
with $0\le h=h(E)\le \sqrt{n}$. We shall consider the case where $0\le E<\sqrt{k/n}$. The case where $E$ lies in the exponentially decaying region of $E_k(x)$ is simpler and follows by a similar analysis.

First, note that the estimate (\ref{eqn:Rdiag1}) is efficient only when $h\ge 1$. Consider the partition of the region $(E,\sqrt{k/n})$ into the intervals 
\[U_p = (\sqrt{k/n}-pk^{-1/6}n^{-1/2},\sqrt{k/n}-(p-1)k^{-1/6}n^{-1/2}), \quad 1\le p \le hk^{1/6}.\]
Note that each $U_p$ is an interval of length $k^{-1/6}n^{-1/2}$. By the local semicircle law for the zero temperature case, the number of eigenvalues in an interval $U_p$ is bounded by  
\[C_cn^{-1/2}k^{-1/6},\] provided $U_p$ is at distance greater than $c>0$ from $1$. Moreover, we have 
\[E_k(x)^2 \le C k^{-1/6}p^{-1/2}, \quad x\in U_p,\] 
and 
\[|\lambda_m-z| \ge |hn^{-1/2}-pk^{-1/6}n^{-1/2}|+n^{-1/2+\epsilon}.\]  So the contribution to the sum in equation (\ref{eqn:Rdiagtemp}) due to eigenvalues $\bar{\lambda}_m$ lying in the region $(E,\sqrt{k/n})$ is bounded up to a constant by
\[
k^{-1/6}n^{-1/2}\sum_{p=1}^{hk^{1/6}} k^{-1/6}p^{-1/2}(|hn^{-1/2}-pk^{-1/6}n^{-1/2}|+n^{-1/2+\epsilon})^{-1}\]
Simplifying the above sum, and letting $h = k^{\alpha}$ for some $0\le \alpha \le 1/2$, we have the bound
\[k^{-1/3-\alpha} \sum_{p=1}^{k^{\alpha+1/6}}p^{-1/2}(|1-pk^{-1/6-\alpha}|+n^\epsilon k^{-\alpha})^{-1}.\]
The sum over $p$ is bounded by $C(\log k) k^{\alpha/2+1/12}$, and so the above quantity is bounded by $C\log n\cdot k^{-1/4-\alpha/2} \le h^{-1}$, since $-1/4\le -\alpha/2$.  

Turning to the contribution to the sum in (\ref{eqn:Rdiagtemp}) from eigenvalues $\bar{\lambda}_m$ in the region $(-1,E)$, we have that 
\[E_k(x)^2 \le k^{-1/4}h^{-1/2} \le h^{-1},\] for $h \le k^{1/2}$. Combined with the fact that $\sum \frac{1}{|\lambda_m-z|} \le Cn\log n$ (see Lemma \ref{lem:absest} below), this establishes the estimate (\ref{eqn:Rdiag1}). The alternative bound $|R_{kk}| \leq n^{1/2-\epsilon}$ follows from Cauchy-Schwarz inequality, noting that the eigenvectors have norm $1$.

Finally, to prove the bound for $R_{kk}$ when $h<n^\eta$, $0<\eta<1/2$, note that in this case
\[k = En+O(n^{1/2+\eta}).\]
The result then follows from \eqref{eqn:hermdecay} and \eqref{eqn:Rdiagtemp}  by a dyadic decomposition as previously. As noted above, the implicit constants depend on $\eta$, but we only used this bound for the fixed choice $\eta=1/4$.

We will continue by proving the bound for $R_{jk}$.  To simplify notation, we shall assume that $E = \Re z = 0$, the case for nonzero $E$ follows similarly. Unlike in the case of the diagonal entries $R_{kk}$, we will exploit cancellation.

We write $R_{kl}$ as
\[R_{kl} = \sum_{m=1}^n \frac{u_m(k)u_m(l)}{\lambda_m-z}.\]
Applying summation by parts, we rewrite the above as
\begin{equation} \label{eqn:sumbyparts}
\sum_{m=1}^{n-1} \left(\frac{1}{\lambda_m-z}-\frac{1}{\lambda_{m+1}-z} \right)\sum_{t=m}^n u_t(k)u_t(l).
\end{equation}
Note that:
\begin{equation}
\label{eqn: offdiagonalsum}
\sum_{t=m}^n u_t(k)u_t(l) = \sum_{t=m}^n \frac{E_k(\lambda_t)E_l(\lambda_t)}{n \left(E_{n-1}(\lambda_t)\right)^2}.
\end{equation}

From here on, we shall assume without loss of generality that $k > l$. We will concentrate on the contribution to (\ref{eqn: offdiagonalsum}) from eigenvalues $|\lambda_t| < (1+\mu)\sqrt{l}$. The sum over $\lambda_t$ lying in the region where $E_l(\cdot)$ decays exponentially is simpler to handle, and uses similar techniques. The following is a key lemma:

\begin{lemma} \label{lem:keybulk}
Given $1 \leq l<k \leq n$, let $m$ be such that: 
\[\lambda_m \in (-\sqrt{l}(1-\mu), \sqrt{l}(1-\mu)),\] for some fixed $0<\mu<1$ independent of $k,l,n$. Then, for any $\tilde{m}$ such that 
\[\lambda_{\tilde{m}} \in (-\sqrt{l}(1-\mu), \sqrt{l}(1-\mu)),\]
we have:
\begin{equation} \label{eqn:osscest}
\left|\sum_{t=m}^{\tilde{m}} \frac{E_k(\lambda_t)E_l(\lambda_t)}{nE_{n-1}^2(\lambda_t)}\right| \leq Ck^{-1/4}l^{-1/4}(\sqrt{k}-\sqrt{l})^{-1}.
\end{equation}
\end{lemma}
\begin{proof}[Proof of lemma \ref{lem:keybulk}]
Since
\[|\lambda_t| \le (1-\mu)\sqrt{l} \le (1-\mu)\sqrt{k},\]
equation (\ref{eqn:oss}) implies:
\begin{displaymath}
\begin{split}
E_k(\lambda_t) &= \sqrt{\frac{2}{\pi}} k^{-1/4}\left(1-\frac{\lambda_t^2}{k}\right)^{-1/4} \cos\left(\sqrt{k}\pi \int_{\lambda_m}^{\lambda_t} \rho_{sc}(y/\sqrt{k}) \,\mathrm{d}y + \frac{1}{2}\arcsin(\lambda_t/\sqrt{k}) + \varphi_{m,k}\right) \left(1+O\left(\frac{1}{k}\right)\right)\\
& + \sin\left(\sqrt{k}\pi \int_{\lambda_m}^{\lambda_t} \rho_{sc}(y/\sqrt{k})\, \mathrm{d}y + \frac{1}{2}\arcsin(\lambda_t/\sqrt{k}) + \varphi_{m,k}\right)O\left(\frac{1}{k}\right),
\end{split}
\end{displaymath}
where 
\begin{align*}\varphi_{m,k} &= \sqrt{k}\pi\int^{\lambda_m}_{\sqrt{k}}\rho_{sc}(y/\sqrt{k})\,\mathrm{d}y\\
&=  k\pi\int^{\frac{\lambda_m}{\sqrt{k}}}_1\rho_{sc}(s)\,\mathrm{d}s.
\end{align*}
A similar expression holds for $E_l(\lambda_t)$.  For $m$ fixed, we will expand $E_k(\lambda_t)E_l(\lambda_t)$ and $(E_{n-1}(\lambda_t))^2$ around $\lambda_m$. By (\ref{eqn:zeroloc}), we have
\begin{displaymath}
\lambda_{t} = \sqrt{n}\left(\zeta\left(\frac{6m-3}{6n}+\frac{6(t-m)}{6n} + \frac{1}{2\pi n}\arcsin\zeta((m+(t-m))/n)\right)\right) + O_\mu(n^{-2}).
\end{displaymath}
By Taylor expansion around $\lambda_m$, we have
\begin{equation}\label{eqn: lambdaexpansion}
\lambda_{t} = \lambda_m + \alpha_1 \frac{t-m}{\sqrt{n}} + O_\mu(n^{-3/2}),
\end{equation}
for $0\le t-m < n.$ The constant 
\[\alpha_1 = \zeta'\left(\frac{6m-3}{6n}+\frac{1}{2n\pi}\arcsin\zeta(m/n)\right)\left(1+\frac{1}{2n\pi}\frac{\zeta'(m/n)}{\sqrt{1-\zeta^2(m/n)}}\right)\] 
is independent of $t$. $|\alpha_1|$ is bounded above and below, uniformly for $m$ such that $\lambda_m \in (-(1-\mu)\sqrt{n},(1-\mu)\sqrt{n})$.  Combining the above and using the product to sum formula $\cos a\cos b =(1/2)\cdot \cos (a+b)+\cos(a-b)$, we have the estimate
\begin{equation*}
\begin{split}
E_k(\lambda_t)E_l(\lambda_t) &= Ck^{-1/4}l^{-1/4}\cos(\sqrt{k}\pi\int_{\lambda_m}^{\lambda_t} \rho_{sc}(y/\sqrt{k}) \,\mathrm{d}y-\sqrt{l}\pi\int_{\lambda_m}^{\lambda_t} \rho_{sc}(y/\sqrt{l}) \,\mathrm{d}y \\ 
&+ \frac{1}{2}\arcsin(\lambda_t/\sqrt{k}) -  \frac{1}{2}\arcsin(\lambda_t/\sqrt{l}) + \varphi_{m,k}-\varphi_{m,l})\left(1+O\left(\frac{1}{l}\right)\right)\\
&+ \left(\cos(\cdot) \text{ term with sum of arguments}\right).
\end{split}
\end{equation*}
The final term denotes an expression formally similar to the initial part of the right side of the equation, except that in the argument of the cosine, all $-$ signs are replaced by $+$ . We will denote $\varphi_m=\varphi_{m,k}-\varphi_{m,l}$; this term is independent of $t$. Taylor expansion of the integrals and $\arcsin$ functions in the expression above yields:
\begin{multline*}
Ck^{-1/4}l^{-1/4}\Big(\cos\Big(\pi(\sqrt{k}\rho_{sc}(\lambda_m/\sqrt{k}) - \sqrt{l}\rho_{sc}(\lambda_m/\sqrt{l}) +\frac{1}{2}(\arcsin(\lambda_m/\sqrt{k})-\arcsin(\lambda_m/\sqrt{l})))\cdot (\lambda_t-\lambda_m)\\ 
+ O((\sqrt{k}-\sqrt{l})/\sqrt{k})\cdot (\lambda_t-\lambda_m)^2 +\varphi_m \Big)\Big)\left(1+O\left(\frac{1}{l}\right)\right).
\end{multline*}
The constants implicit in the $O(\cdot)$ terms are bounded uniformly in $m$ such that $\lambda_m \in ((-1+\mu)\sqrt{l}, (1-\mu)\sqrt{l})$. Using the expansion (\ref{eqn: lambdaexpansion}) of $\lambda_t$ around $\lambda_m$, we are lead to the following approximate expression for the product $E_k(\lambda_t)E_l(\lambda_t)$:
\begin{multline}\label{eqn: approxprod}
Ck^{-1/4}l^{-1/4}\cos\left(2 \alpha_1(\sqrt{k}-\sqrt{l})\cdot \left(\frac{t-m}{\sqrt{n}}\right)+\varphi_m\right)\\ 
+Ck^{-1/4}l^{-1/4} \cdot \left( O((\sqrt{k}-\sqrt{l})/\sqrt{k}) \cdot n^{-3/2})+O(((\sqrt{k}-\sqrt{l})/\sqrt{k})\cdot n^{-1}(t-m)^2)\right),
\end{multline}
where $\alpha_1$ is the constant independent of $t$ in (\ref{eqn: lambdaexpansion}). We have suppressed lower order terms in the final displayed equation. In particular, we have omitted the second cosine term referred to above, involving sums instead of differences in the argument. When $t$ varies, this term oscillates faster, and the remainder of the argument will make it clear that it is bounded by the contribution from the term in the last equation. 

We wish to use the approximations we have introduced to estimate the sum in (\ref{lem:keybulk}) between $\lambda_m$ and $\lambda_{\tilde{m}}$. Both eigenvalues are assumed to lie in the range $((-1+\mu)\sqrt{l}, (1-\mu)\sqrt{l})$, so that $|\lambda_m-\lambda_{\tilde{m}}|<2\sqrt{l}$. Thus there are $O(\sqrt{nl})$ indices $t$ such that $\lambda_m \le \lambda_t \le \lambda_{\tilde{m}}$. To take advantage of the oscillation of the cosine term in our approximation, we perform the sum ``period by period''.

For fixed $m$, we sum $E_k(\lambda_t)E_l(\lambda_t)/n(E_{n-1}(\lambda_t))^2$ over indices $t$ in the range
\[  m \le t < m+ \frac{\pi\sqrt{n}}{\alpha_1(\sqrt{k}-\sqrt{l})}\equiv m+\Lambda,\] 
corresponding to a full period of the cosine term in (\ref{eqn: approxprod}):
\begin{equation*}
\sum_{t-m=0}^{\lfloor \Lambda \rfloor -1} \frac{E_k(\lambda_t)E_l(\lambda_t)}{n(E_{n-1}(\lambda_t))^2} = \mathrm{I}(m) +\mathrm{II}(m).
\end{equation*}
Here, $\mathrm{I}(m)$ represents the contribution to the sum from the first (oscillatory) term in (\ref{eqn: approxprod}). The sum $\mathrm{II}(m)$, the contribution from the error terms, is readily estimated; by (\ref{eqn: approxprod}) and (\ref{eqn:En-1}), we have for any $m$:
\begin{equation}\label{eqn: twoofm}
|\mathrm{II}| \le Ck^{-1/4}l^{-1/4}(n^{-3/2}k^{-1/2}+(\sqrt{k}-\sqrt{l})^{-2}k^{-1/2}).
\end{equation}
To estimate the sum $\mathrm{I}(m)$, we first consider the case $\Lambda \ge 100$. After using a trigonometric identity, we are faced with sums of the form:
\begin{equation}
\label{eqn: cmsum}
c(m)\sum_{k=0}^{\lfloor \Lambda\rfloor -1} \cos\left(\frac{2\pi}{\Lambda} k\right),
\end{equation}
where $c(m)$ depends on the phase $\varphi(m)$, and $|c(m)|\le 1$. The sum is bounded uniformly in $\Lambda$. Moreover, the sum remains bounded when the upper limit of the sum is varied by less than $5$:
\[\left| \sum_{k=0}^{\lfloor \Lambda\rfloor \pm r} \cos\left(\frac{2\pi}{\Lambda} k\right)\right| \le C, \quad r\le5. \]
By choosing $0\le r \le 5$ appropriately the expression in (\ref{eqn: cmsum}) can be made either positive or negative.  There are $O(\sqrt{l}(\sqrt{k}-\sqrt{l}))$ periods separating $\lambda_m$ and $\lambda_{\tilde{m}}$. We apply the above approximation repeatedly, using the bound 
(\ref{eqn: twoofm}), and adjusting the upper limit of the sum over each period so as to make the signs of $\mathrm{I}$ alternate. The estimate (\ref{eqn:osscest}) thus follows in the case $\Lambda \ge 100$. In case $\Lambda < 100$, we must modify the previous proof somewhat. The error terms in (\ref{eqn: approxprod}) can be treated as before, but when dealing with the oscillatory term, we are no longer assured that a small change in the upper limit will change the sign of the trigonometric sum in (\ref{eqn: cmsum}). The phase term $\varphi_m$ must be taken into account. When $m' > m$, we have:
\begin{align*}
\varphi_{m'} &= \varphi_{m} + \sqrt{k}\pi\int^{\lambda_{m'}}_{\lambda_m}\rho_{sc}(y/\sqrt{k})\,\mathrm{d}y - \sqrt{l}\pi\int^{\lambda_{m'}}_{\lambda_m}\rho_{sc}(y/\sqrt{l})\,\mathrm{d}y.\\
&= \varphi_m  + 2(\sqrt{k}-\sqrt{l})(\lambda_{m'}-\lambda_m)+O\left(\frac{\sqrt{k}-\sqrt{l}}{\sqrt{k}\sqrt{l}}(\lambda_{m'}-\lambda_m)^3\right).
\end{align*}
Using this expansion when $m$ and $m'$ define the limits of two consecutive periods, we can 
ensure that the oscillatory term in (\ref{eqn: approxprod}), when summed over different periods, remains bounded by $Ck^{-1/4}l^{-1/4}n^{-1/2}$. The approximation of the $\varphi_{m'}$ introduces an error that is controlled similarly to the terms in (\ref{eqn: twoofm}).
\end{proof}

In the transition region, the behavior is similar:
\begin{lemma} \label{lem:keytrans}
Given $1 \leq l<k \leq n$, let $m$ be such that $\lambda_m \in (-\sqrt{l}(1-\mu), \sqrt{l}(1-\mu))$ for some fixed $\mu$ independent of $k,l,n$, then there exists $\tilde{m}$ such that $\lambda_{\tilde{m}} \in (-\sqrt{l}(1-\mu), \sqrt{l}(1-\mu))$
\begin{equation} \label{eqn:transcest}
\left|\sum_{t=m}^{\tilde{m}} \frac{E_k(\lambda_t)E_l(\lambda_t)}{n(E_{n-1}(\lambda_t))^2}\right| \leq Ck^{-1/4}l^{-1/4}n^\eta(\sqrt{k}-\sqrt{l})^{-1}
\end{equation}
for any $\eta > 0$.
\end{lemma}

\begin{rmk}
We will only prove the lemma for the right edge.  The proof of the statement for the left edge at $-1$ is identical.
\end{rmk}

\begin{proof} [Proof of lemma \ref{lem:keytrans}] The proof uses ideas similar to those in the proof of (\ref{lem:keybulk}): we replace $E_l$ and $E_k$ by appropriate asymptotic expressions and thus reduce the problem to estimating some trigonometric sums with slowly varying frequencies and phases. We will omit some details and give the orders of magnitude of the approximations and error terms involved.

We use asymptotics for the Airy function and its derivative (see \cite{AS}, Ch. 10, section 4):
\begin{equation} \label{eqn:Airy}
Ai(x) = \frac{1}{2\sqrt{\pi}x^{1/4}}e^{-\frac{2}{3}x^{3/2}}(1+O(x^{-3/2}))
\end{equation}
and
\begin{equation} \label{eqn:Airyd}
Ai'(x) = \frac{-1}{2\sqrt{\pi}}x^{1/4}e^{-\frac{2}{3}x^{3/2}}(1+O(x^{-3/2})).
\end{equation}
We first establish an estimate for the magnitude of $E_l$ in the transition region.  If $\lambda = l^{1/2}-l^{-1/2+4\alpha}$ for some $\alpha$, then have 
\begin{align*}
|E_l(\lambda)| &\le Cl^{-1/4}\left(1-\frac{\lambda}{\sqrt{l}}\right)^{-1/4}\\
&= Cl^{-\alpha}.
\end{align*}
Using the asymptotics for the Hermite functions in the transition region (\ref{eqn:trans}) and the asymptotics for the Airy function (\ref{eqn:Airy}) for $x$ close to $\lambda$, we find:
\begin{align*}
E_l(x) &=  Cl^{-\alpha} \sin\left(\frac{3\pi}{2}\int_1^{x/\sqrt{l}}\rho_{sc}(y) \,\mathrm{d}y \cdot l\right)\\ 
&= C l^{-\alpha}\sin\Big(l(\rho_{sc}(\lambda/\sqrt{l}))\cdot \frac{x-\lambda}{\sqrt{l}}+O(\rho'_{sc}(\lambda/\sqrt{l}))\cdot \left(\frac{x-\lambda}{\sqrt{l}}\right)^2\Big).
\end{align*}

We first consider the situation where $\sqrt{k} - \sqrt{l} > \mu\sqrt{l}$.  Fix an eigenvalue $\lambda_m$ with 
\[(1-\mu)\sqrt{l}  < \lambda_m < \sqrt{l}-l^{-1/6}.\] Near $\lambda_m$, we have:
\begin{equation} \label{eqn:tempkltrans}
\begin{split}
E_k(\lambda_t)E_l(\lambda_t) &= k^{-1/4}l^{-1/4}\cos\left(\sqrt{k}\pi \int_{\lambda_m}^{\lambda_t} \rho_{sc}(y/\sqrt{k})\,\mathrm{d}y + \frac{1}{2}\arcsin(\lambda_t/\sqrt{l})\right) \\ 
& \times \left(1+\frac{\lambda_t}{\sqrt{l}}\right)^{1/4}\left|1-\frac{\lambda_t}{\sqrt{l}}\right|^{-1/4}(f(\lambda_t/\sqrt{l}))^{-1/4}Ai(f(\lambda_t/\sqrt{l})).
\end{split}
\end{equation}
Here again we omit lower order terms, as they are bounded by the expression above. Write 
\[\lambda_m = \sqrt{l}-l^{-1/2+4\alpha},\] 
for some $\frac{1}{12} \leq \alpha \leq \frac{1}{4}$ and $l^{-1/2+4\alpha} < \mu l^{1/2}$. By a previous remark, the magnitude of $E_l$ around $\lambda_m$ is $Cl^{-\alpha}$. For $\lambda_t$ near $\lambda_m$, the approximate expression for $E_l$ found above becomes
\begin{equation*}
E_l(\lambda_t) = l^{-\alpha}\cos\left(l^{2\alpha}(\lambda_t-\lambda_m)+O(l^{1/2-2\alpha}(\lambda_t-\lambda_m)^2)+\psi_{t,m}\right),
\end{equation*}
where $\psi_{t,m}$ denotes a phase term independent of $t$. From equations (\ref{eqn:tempkltrans}), (\ref{eqn:zeroloc}) and (\ref{eqn:En-1}), it follows that there exists $\tilde{m}$ of order $m+O(\sqrt{n}(k^{1/2}-l^{2\alpha})^{-1})$ such that the following estimate holds:
\begin{equation*}
\left|\sum_m^{\tilde{m}}\frac{E_k(\lambda_t)E_l(\lambda_t)}{n(E_{n-1}(\lambda_t))^2}\right| \leq Ck^{-1/4}l^{-\alpha}l^{1/2-2\alpha}(k^{1/2}-l^{2\alpha})^{-3}.
\end{equation*}
Consider a partition of the interval $(\sqrt{l}-l^{-1/2+4\alpha},\sqrt{l}-l^{-1/6})$ into subintervals of the form 
\[(\sqrt{l}-l^{-1/2+4\alpha-(p-1)\eta},\sqrt{l}-l^{-1/2+4\alpha-p\eta}), \quad p=1,2,\ldots .\]
Here $\eta>0$ is some parameter to be determined. There are $O(\eta^{-1})$ such intervals.  The length of each interval is $l^{-1/2+4\alpha}$ and so the sum over $\lambda_m$ belonging to such an interval is bounded, up to a constant factor, by 
\[k^{-1/4}l^{-\alpha}l^{-1/2+4\alpha}l^{1/2-2\alpha}(k^{1/2}-l^{2\alpha})^{-2}.\]
This last expression is $Ck^{-1/4}l^{-1/4}(\sqrt{k}-\sqrt{l})^{-1}$. In summary, for any $\lambda_m$ such that $\lambda_m<(1-\mu) \sqrt{l}$, there exists some $\lambda_{\tilde{m}}$ near $\sqrt{l}-l^{-1/6}$ such that 
\begin{displaymath}
\sum_m^{\tilde{m}}\frac{E_k(\lambda_t)E_l(\lambda_t)}{n(E_{n-1}(\lambda_t))^2} \leq C\frac{l^\eta}{\eta}l^{-1/4}k^{-1/4}(\sqrt{k}-\sqrt{l})^{-1}.
\end{displaymath}
The contribution to the sum from eigenvalues $\lambda_t$ in the region between $\sqrt{l}-l^{-1/6}$ and $\sqrt{l}$ can be trivially bounded by $Cl^{-1/12}l^{-1/6}k^{-1/4}k^{-1/2}$ using the amplitude bound for $E_l$ in the region and the oscillation from $E_k$.  This establishes the case where $\sqrt{k}-\sqrt{l} > \mu\sqrt{l}$.

Next, we consider the case where 
\[\sqrt{k}-\sqrt{l} = l^\beta,\]
for some $\beta$ such that $l^\beta < \mu l^{1/2}$.  In this case, the sum involves contributions from products $E_k(\lambda_m)E_l(\lambda_m)$ such that $\lambda_m$ lies in the transition region for both $E_l$ and $E_k$. Suppose 
\[ \lambda_m = \sqrt{l}-l^{-1/2+4\alpha},\] 
for some $\alpha$ such that $|E_l(\lambda_m)| = l^{-\alpha}$. Since $k > l$, we have the following estimate (writing $m=l^\beta+l^{-1/2+4\alpha}$ to simplify notations):
\begin{equation} \label{eqn:Ekedge}
E_k(x) = Ck^{-1/8}m^{-1/4} \sin\left(k^{1/4}m^{1/2}(x-\lambda) + O(k^{1/4}m^{-1/2})\cdot (x-\lambda)^2\right).
\end{equation}
A similar expression holds for $E_l$. Thus, for $\lambda_t$ near $\lambda_m$, the product $E_k(\lambda_t)E_l(\lambda_t)$ can be replaced by a sum of trigonometric functions oscillating at frequency $k^{1/4}m^{1/2}-l^{2\alpha}$, at the expense of introducing an error of order 
\[O(k^{1/4}m^{-1/2}-l^{1/2-2\alpha})\cdot (x-\lambda)^2.\] Using the error estimates (\ref{eqn:zeroloc}) and (\ref{eqn:Ekedge}) and summing over the eigenvalues in a region of size $O(l^{-1/2+4\alpha+\eta})$ (containing $O(l^{-1/2+4\alpha+\eta}(k^{1/4}m^{1/2}-l^{2\alpha}))$ periods) will introduce an error of 
\begin{displaymath}
Ck^{-1/8}m^{-1/4}l^{-1/2+3\alpha+\eta}\cdot \frac{k^{1/4}m^{-1/2}-l^{1/2-2\alpha}}{(k^{1/4}m^{1/2}-l^{2\alpha})^2}.
\end{displaymath}
We consider two cases: (\emph{i}) $\beta \geq -\frac{1}{2}+4\alpha$ and (\emph{ii}) $\beta < -\frac{1}{2}+4\alpha$.  For the first case, the above expression simplifies to (recall the definition of $m$):
\[
Ck^{-1/8-1/4} l^{-3\beta/4-1/2+3\alpha+\eta} \cdot (k^{1/4}l^{-\beta/2}-l^{1/2-2\alpha}) \le C k^{-1/4}l^{-1/4}l^{-\beta+\eta}.
\]
In case (\emph{ii}), we have the bound:
\[  Ck^{-1/4}k^{-1/8}l^{1/8}l^{1/4-4\alpha}l^{\eta}.
\]
Since $\alpha > \frac{1}{8}+\frac{\beta}{4}$, the above expression is bounded by $Ck^{-1/4}l^{-1/4}l^{-\beta}l^{\eta}$.  Finally, since $l^{-\beta} = (\sqrt{k}-\sqrt{l})^{-1}$, proceeding as in the case of $\sqrt{k}-\sqrt{l} > \mu\sqrt{l}$, we arrive at the desired estimate. 
\end{proof}
The contribution from eigenvalues $\lambda_m$ in the region ($\lambda_m > \sqrt{l}+l^{-1/6+\epsilon}$) where $E_l$ decays exponentially is easier to handle than the cases treated above. The case $\sqrt{l}<|\lambda_m|<\sqrt{l}+l^{-1/6+\epsilon}$ can be estimated in a similar manner to the case $(1-\mu)\sqrt{l}<|\lambda_m|<\sqrt{l}$. Combining the two lemmas above and first summing over $\lambda_m$ in the two transition regions, and then the oscillatory region, we have established the following proposition:
\begin{proposition}
For any $m$, we have the bound
\begin{equation*}
\sum_m^n\frac{E_k(\lambda_t)E_l(\lambda_t)}{n(E_{n-1}(\lambda_t))^2} \leq Cn^{\eta}l^{-1/4}k^{-1/4}(\sqrt{k}-\sqrt{l})^{-1}
\end{equation*}
for any $\eta > 0$ and some constant $C$.
\end{proposition}

We substitute the estimate in the previous proposition into equation (\ref{eqn:sumbyparts}) to obtain the following bound for $R_{kl}$ (recall that we are assuming $\Re z = 0$):
\begin{equation*}
\begin{split}
|R_{kl}| &\leq Cn^{\eta}l^{-1/4}k^{-1/4}(\sqrt{k}-\sqrt{l})^{-1} \sum_m\left|\frac{1}{\lambda_m-z}-\frac{1}{\lambda_{m+1}-z}\right|.
\end{split}
\end{equation*}
To estimate the above, perform a dyadic decomposition around $\Re z$. Consider the eigenvalues $\lambda_m$ such that 
\[ 2^p n^{-1/2+\epsilon}<|\lambda_m| < 2^{p+1} n^{-1/2+\epsilon}.\]  
Their number is  $O(2^{p}n^{-1/2+\epsilon})$ and the factor $|\frac{1}{\lambda_m-z}-\frac{1}{\lambda_{m+1}-z}|$ is $O(n^{-2\epsilon}2^{-2p})$.  Therefore, the contribution to the sum from such eigenvalues is bounded by $Cn^{1/2-\epsilon}2^{-p}$.  Summing over $p$, we find: 
\[ |R_{kl}|\le Cn^{\eta}n^{1/2-\epsilon}l^{-1/4}k^{-1/4}(\sqrt{k}-\sqrt{l})^{-1}.\] 
Taking, say, $\eta = \epsilon/10$, we have established Proposition \ref{prop:R}.
\end{proof}

\section{From 1/2 to 1: Inductive arguments}
\label{sec:one}
In this section, we improve the semicircle law from the level of $n^{-1/2+\epsilon}$ to the optimal level of $n^{-1+\epsilon}$ by an inductive argument.

\begin{theorem} \label{thm:main}
Let $s(z)$ be the Stieltjes transform of the measure induced by the eigenvalues of the normalized Gaussian $\beta$-ensemble, $\frac{1}{\sqrt{n}}A_{n,\beta}$.  Let $s_{sc}(z)$ be the Stieltjes transform of the semicircle law.  Then with overwhelming probability:
\begin{equation}
\sup_{z\in D_{\epsilon,\delta}}|s(z)-s_{sc}(z)| = o(1)
\end{equation}
where for $\epsilon>0$ and $0<\delta<1$, the domain $D_{\epsilon, \delta}$ is defined as
\begin{displaymath}
D_{\epsilon,\delta} := \{z: \Im z > n^{-1+\epsilon}, -1+\delta < \Re z < 1-\delta\}.
\end{displaymath}
\end{theorem}

Before we start we will need two facts about the tridiagonal models and Stieltjes transforms.  The first can be found in \cite{D}:

\begin{proposition} \label{prop:indep}
The tridiagonal model $A_\beta$ can be diagonalized as $A = Q\Lambda Q^*$, such that the first row of $Q$ is independent of $\Lambda$ and consists of independent entries of $\chi_\beta$-distribution normalized to unit norm.
\end{proposition}

The second lemma establishes the link between control of Stieltjes transform and control of distribution of eigenvalues and can be found in \cite{TV}:
\begin{lemma} \label{lem:tv64}
Let $1/10 \geq \eta \geq 1/n$ and $L,\epsilon,\tau > 0$.  Suppose that one has the bound
\begin{displaymath}
|s_n(z) - s(z)| \leq \tau
\end{displaymath}
with overwhelming probability for all $z$ with $|\Re z|<L$ and $\Im z > \eta$.  Then for any interval $I$ in $[-L+\epsilon,L-\epsilon]$ with $|I|\geq \max(2\eta,\frac{\eta}{\tau}\log\frac{1}{\tau}$), one has
\begin{displaymath}
\left|N_I - n\int_I \rho_{sc}(y) dy\right| <_\epsilon \tau n|I|
\end{displaymath}
with overwhelming probability, where $N_I$ denotes the number of eigenvalues in $I$.
\end{lemma}

Let us introduce some terminology: we will say the \emph{semicircle law holds at level} $\alpha$ if, for $z \in D_{1-\alpha, \delta}$, and any $c>0$
\[|s_\beta(z)-s_{sc}(z)|<c\]
for sufficiently large $n$, with overwhelming probability. We first prove the following simple lemma:
\begin{lemma} \label{lem:absest}
Suppose the semicircle law holds at level $\Im z > n^a$ (or equivalently by lemma \ref{lem:tv64} the Stietljes transform is close to that of the semicircle law at the level $n^a$) for some $-1 < a < 0$, then we have
\begin{equation}
\frac{1}{n}\sum_j \frac{1}{|\lambda_j-z|^2} \leq C (\Im z)^{-2}n^a \log n
\end{equation}
for any $z$ such that $n^{-1} < \Im z < n^a$.
\end{lemma}

\begin{proof}
We will in fact show the following inequality:
\begin{equation} \label{eqn:absres}
\frac{1}{n}\sum_j \frac{1}{|\lambda_j-z|} \leq C (\Im z)^{-1}n^a \log n
\end{equation}
The lemma then follows from the trivial bound:
\[\frac{1}{|\lambda_j-z|} \leq (\Im z)^{-1},\]
holding for all $j$.
Suppose the semicircle law holds at level $n^{-1} < \eta < 1/10$, that is:
\[|s(z) - s_{sc}(z)| = o(1)\]
for any $z$ such that $\Im z = \eta$.  First notice that we have the following trivial bound for any $w$
\[
\sum_j \frac{1}{|\lambda_j-w|} \leq \frac{\eta}{\Im w} \sum_j \frac{1}{|\lambda_j-z|},
\]
where $\Im z = \eta$.  It remains to show that
\[\sum_j\frac{1}{|\lambda_j-z|} \leq n \log n\]
for such $z$.
Let $E = \Re z$, and consider a dyadic partition around $E$:
\begin{displaymath}
U_p := \{j: 2^{p-1}\eta \leq |\lambda_j-E| \leq 2^p\eta\} 
\end{displaymath} 
for $p = 1, \hdots, \log \eta^{-1}$
\begin{displaymath}
U_0 := \{j: |\lambda_j-E| \leq \eta\};
\end{displaymath}
\begin{displaymath}
U_\infty:= \{j: |\lambda_j-E| > 1\}; 
\end{displaymath}
By lemma \ref{lem:tv64}, $|U_p| \leq 2^pn\eta$ and over the set $U_p$, $\frac{1}{|\lambda_j-E|} \leq 2^{1-p}\eta^{-1}$, thus
\begin{displaymath}
\sum_{j\in U_p}\frac{1}{|\lambda_j-E|} \leq 2n.
\end{displaymath}
Summing over $n$, we obtain the desired bound after noting that $\log \eta^{-1} \leq \log n$.
\end{proof}

The proof of the theorem relies on the next two propositions.
\begin{proposition} \label{prop:induct1}
Suppose the semicircle law holds at level $a$ for some $-1 < a < 0$, then 
\[
|R^\beta_{11}(z) - s_{sc}(z)| = o(1)
\]
for $z$ such that $\Im z > n^{(a-1)/2+\delta}$ for any $\delta > 0$ with overwhelming probability.
\end{proposition}
\begin{proposition} \label{prop:induct2}
Suppose $|R^\beta_{11}(z) - s_{sc}(z)| = o(1)$ for $z$ such that  $\Im z > n^a$ for some $-1 < a < 0$, then we have an improved semicircle law, i.e.
\[
\left|\frac{1}{n} \sum_j \frac{1}{\lambda_j-z} - s_{sc}(z)\right| = o(1)
\]
for $z$ such that $\Im z > n^{(a-1)/2+\delta}$ for any $\delta > 0$ with overwhelming probability.
\end{proposition}

\begin{proof} [Proof of Theorem \ref{thm:main}]
By Proposition \ref{prop:1/2}, we have a local semicircle law at the level of $-1/2+\epsilon$.  Applying Propositions \ref{prop:induct1} and \ref{prop:induct2} repeatedly will yield the desired conclusion.
\end{proof}
What remains is the proof of the two propositions above:
\begin{proof} [Proof of Proposition \ref{prop:induct1}]
By Schur's complement, we have the following relationship
\[R^\beta_{11}(z) = \frac{1}{\frac{1}{\sqrt{n\beta}}a_{11}-z-\frac{1}{n\beta}b^2_{12}\hat{R}_{11}}\]
where $a_{11}$ is the $(1,1)$-entry of our (symmetric) tridiagonal matrix $A$, normally distributed with mean $0$ and variance $1$ and $b_{1,2}$ is the $(1,2)$-entry of $A$, which follows a $\chi$-distribution with $(n-1)\beta$ degrees of freedom.  $\hat{R}$ is the resolvent of the matrix obtained by removing the first row and first column of $A$. We will denote the latter matrix by $\hat{A}$.  

Firstly, notice that $a_{11}, b_{12}$ and $\hat{R}$ are independent, and that $\frac{1}{\sqrt{n\beta}}|a_{11}| < n^{-1/2+c}$ with overwhelming probability for any $c > 0$.  Thus, we have
\begin{equation} \label{eqn:denom}
\frac{1}{\sqrt{n\beta}}a_{11}-z-\frac{1}{n\beta}b^2_{12}\hat{R}_{11} = -z-\frac{1}{n\beta}\mathbb{E}_bb^2_{12}\hat{R}_{11} + \frac{1}{n\beta}(b^2_{12}-\mathbb{E}_bb^2_{12})\hat{R}_{11}+O(n^{-1/2+c}).
\end{equation}
For the second term on the right hand side, we have $\mathbb{E}_bb^2_{12} = (n-1)\beta$, so 
\[\frac{1}{n\beta}\mathbb{E}_bb^2_{12}\hat{R}_{11} = \frac{n-1}{n}\hat{R}_{11}.\]
Let $(\hat{q}_1,\hdots\hat{q}_{n-1})$ be the first row of the eigenvectors for $\hat{A}$, $\hat{\lambda}_j$ be the eigenvalues and write
\[\hat{R}_{11} = \sum_j \frac{\hat{q}_j^2}{\hat{\lambda}_j-z}.\]
Similarly,
\[R_{11} = \sum_j \frac{q_j^2}{\lambda_j-z}.\]
We would like to compare $R_{11}$ and $\hat{R}_{11}$, to this end we write
\begin{equation} \label{eqn:compare}
\begin{split}
R_{11}-\frac{n-1}{n}\hat{R}_{11} &= \sum_j\frac{q_j^2}{\lambda_j-z}-\mathbb{E}_q\sum_j\frac{q_j^2}{\lambda_j-z}\\  
&+ \mathbb{E}_q\sum_j\frac{q_j^2}{\lambda_j-z} - \frac{n-1}{n}\mathbb{E}_{\hat{q}}\sum_j\frac{\hat{q}_j^2}{\hat{\lambda}_j-z}\\
&+ \frac{n-1}{n}\left(\mathbb{E}_{\hat{q}}\sum_j\frac{\hat{q}_j^2}{\hat{\lambda}_j-z} - \sum_j\frac{\hat{q}_j^2}{\hat{\lambda}_j-z}\right).
\end{split}
\end{equation}
By proposition \ref{prop:indep}, $q$ and $\lambda_j$ are independent and so are $\hat{q}$ and $\hat{\lambda}_j$, 
\begin{displaymath}
 \sum_j\frac{q_j^2}{\lambda_j-z}-\mathbb{E}_q\sum_j\frac{q_j^2}{\lambda_j-z} = \sum_j\frac{q_j^2-\frac{1}{n}}{\lambda_j-z}.
\end{displaymath}
Since $q$ and $\lambda_j$ are independent, we can condition on $\lambda_j$ and apply McDiarmid's inequality (see the appendix) to the sum to conclude that the term on the right hand side is bounded with overwhelming probability by 
\[ \frac{1}{n^{1-c}}\left(\sum_j\frac{1}{|\lambda_j-z|^2}\right)^{1/2}.\] 
By Lemma \ref{lem:absest}, this last quantity is in turn bounded by $C (\Im z)^{-1}n^{-1/2+a/2+c} \log n$ for some constant $C$ and any $c > 0$.  The fifth and sixth term of equation (\ref{eqn:compare}) are similarly bounded.
For the middle two terms, we used the interlacing property of the eigenvalues of a matrix and its minor, to obtain that the term is bounded by $\frac{1}{n\eta}$, where $\eta = \Im z$.

The last thing we have to control is the term $\frac{1}{n\beta}(b^2_{12}-\mathbb{E}_bb^2_{12})\hat{R}_{11}$ in equation (\ref{eqn:denom}).  Since $b^2_{12}$ has the $\chi$-distributions with degree $(n-1)\beta$, $\frac{1}{n\beta}(b^2_{12}-\mathbb{E}_bb^2_{12})$ is bounded by $n^{-1/2+c}$ with overwhelming probability for any $c > 0$. By equation, (\ref{eqn:absres}) and \ref{prop:indep}, $\hat{R}_{11} \le C (\Im z)^{-1}n^{a+c}$ for any $c > 0$.

Putting everything together, we have the following relationship:
\begin{displaymath}
R_{11}(z) = \frac{1}{-z-R_{11}} + O((\Im z)^{-1}n^{-1/2+a/2+c}+(n\eta)^{-1}+(\Im z)^{-1}n^{-1/2+a+c}).
\end{displaymath}
The condition of the proposition guarantees that the error is $o(1)$, and by the standard argument of inspecting the functional equation of the Stieltjes transform of the semicircle law, the proposition is established.
\end{proof}

\begin{proof} [Proof of Proposition \ref{prop:induct2}]
The proof is similar to that of Proposition \ref{prop:induct1}.  By the assumption, it suffices to establish that
\begin{displaymath}
\left|\frac{1}{n}\sum_j\frac{1}{\lambda_j-z}-R^\beta_{11}(z)\right| = o(1)
\end{displaymath}
with overwhelming probability.  The difference inside the absolute value sign is $\sum_j\frac{\frac{1}{n}-q^2_j}{\lambda_j-z}$.  So, again by a concentration of measure argument, and lemma \ref{lem:absest}, the statement is established.
\end{proof}

\section{Acknowledgements}
The authors would like to thank their respective advisors, M. Aizenman and Ya. G. Sinai, for their interest in this work.  We would also like to thank Prof. Sarnak for pointing out important references regarding asymptotics of classical orthogonal polynomials.

\section{Appendix}
Here we prove a slight modification of McDiarmid's inequality used in our inductive argument.

\begin{proposition}(A modification of McDiarmid's inequality.)
Let $X_1,\hdots,X_n$ be independent subgaussian random variables and suppose $F$ is a function of $n$ variables such that there exists an event $\Omega$ with overwhelming probabibility that if $x_1,\hdots,x_n,\tilde{x_i} \in \Omega$, then  
\begin{displaymath}
|F(x_1,\hdots,x_n) - F(x_1,\hdots,x_{i-1},\tilde{x_i},x_{i+1},\hdots,x_n)| \leq c_i
\end{displaymath}
for all $1\leq i \leq n$, and outside of $\Omega$, $F$ is bounded by a polynomial of $n$.  Then for any $\lambda > 0$, one has
\begin{displaymath}
\mathbb{P}(|F(X)-\mathbb{E}(F(X))| \geq \lambda \sigma) \leq C \exp(-c\lambda^2)
\end{displaymath}
for some absolute constant $C,c > 0$, and $\sigma = \sum_{i=1}^nc_i^2$.
\end{proposition}
\begin{proof}
By symmetry, it suffices to show
\begin{displaymath}
\mathbb{P}(F(X)-\mathbb{E}F(X) \geq \lambda \sigma) \leq C\exp(c\lambda^2)
\end{displaymath}
Let $t > 0$ be a parameter to be chosen later. Consider the exponential moment
\begin{displaymath}
\mathbb{E}(\exp(tF(X))|X_1,\hdots,X_{n-1}, \{X_i\}\in \Omega).
\end{displaymath}
Writing $Y = F(X) - \mathbb{E}(F(X)|X_1,\hdots,X_{n-1}, \{X_i\}\in \Omega)$, we can rewrite the above as:
\begin{displaymath}
\mathbb{E}(\exp(tF(Y))|X_1,\hdots,X_{n-1}, {X_i}\in \Omega) \exp(t\mathbb{E}(F(X))|X_1,\hdots,X_{n-1}, \{X_i\}\in \Omega).
\end{displaymath}
By the condition of the theorem, $tY$ fluctuates only by at most $tc_n$ and has mean $0$.  By Hoeffding's lemma, we have
\begin{displaymath}
\mathbb{E}(\exp(tF(Y))|X_1,\hdots,X_{n-1}, \{X_i\}\in \Omega) \leq \exp(O(t^2c_n^2)).
\end{displaymath}
Integrating out the conditioning, we have the bound
\begin{displaymath}
\mathbb{E}(\exp(tF(X)) \leq \exp(O(t^2c_n^2))\mathbb{E}(t(\mathbb{E}F(X)|X_1,\hdots,X_{n-1},\{X_i\}\in \Omega)).
\end{displaymath}
Now the latter expectation is a function of the first $(n-1)$ variables and obeys the same hypothesis, so we can iterate and obtain the bound:
\begin{displaymath}
\exp(\sum_{i=1}^nO(t^2c_i^2)) \mathbb{E}(t(\mathbb{E}F(X)|\{X_i\}\in \Omega))
\end{displaymath}
and by the overwhelming probability condition and the bound outside of the set $\Omega$, we have
\begin{displaymath}
\mathbb{P}(F(X) - \mathbb{E}F(X) \geq \lambda \sigma) \leq \exp(O(t^2\sigma^2)-t\lambda\sigma)
\end{displaymath}
Optimizing in $t$ gives the desired result.
\end{proof}

\end{document}